\documentclass{amsart}

\usepackage{amssymb, array, verbatim, amscd, xypic}
\usepackage{amsmath, amsfonts,amsthm}

\theoremstyle{plain}
\newtheorem{theorem}{Theorem}
\newtheorem{corollary}[theorem]{Corollary}

\newtheorem{proposition}[theorem]{Proposition}

\theoremstyle{definition}

\newtheorem{remark}[theorem]{Remark}

\newcommand{\CC}{{\mathcal C}}

\newcommand{\Q}{{\mathbb Q}}
\newcommand{\Z}{{\mathbb Z}}

\newcommand{\Ima}{\mathrm{Im}}

\newcommand{\Ker}{\mathrm{Ker}}

\newcommand{\Modr}{\mathrm{Mod}\text{-}}

\newcommand{\End}{\operatorname{End}}
\newcommand{\Hom}{\operatorname{Hom}}

\begin{document}


\title[A Baer-Kaplansky theorem]{A Baer-Kaplansky theorem for modules over principal ideal
domains}
\author{Simion Breaz} \thanks{Research supported by the
CNCS-UEFISCDI grant PN-II-RU-TE-2011-3-0065}
\address{Babe\c s-Bolyai University, Faculty of Mathematics
and Computer Science, Str. Mihail Kog\u alniceanu 1, 400084
Cluj-Napoca, Romania}
\email{bodo@math.ubbcluj.ro}




\begin{abstract}
We will prove that if $G$ and $H$ are modules over a principal
ideal domain $R$ such that the endomorphism rings $\End_R(R\oplus
G)$ and $\End_R(R\oplus H)$ are isomorphic then $G\cong H$.
Conversely, if $R$ is a Dedekind domain such that two $R$-modules
$G$ and $H$ are isomorphic whenever the rings $\End_R(R\oplus G)$
and $\End_R(R\oplus H)$ are isomorphic then $R$ is a PID.
%
%
 \end{abstract}


\maketitle
\date{}

\section{Introduction}

The Baer-Kaplansky Theorem, \cite[Theorem 108.1]{Fu2}, states that
two primary abelian groups with isomorphic endomorphism rings are
necessarily isomorphic. This statement was extended to various
classes of modules (abelian groups), e.g. in \cite{Iv}, \cite{Li},
\cite{RW}, \cite{wlf89}, \cite{wlf62}. However straightforward
examples show that in order to obtain such extensions we need to
impose restrictions on these classes. For instance the
endomorphism rings of the Pr\"ufer group $\Z(p^\infty)$ and of the
group of $p$-adic integers $\widehat{\Z}_p$ are both isomorphic to
the ring $J_p$ of $p$-adic integers. This fact suggests that we
need to restrict to some good classes of modules in order to
obtain a Baer-Kaplansky type theorem. Such a result (valid for
torsion-free modules over valuation domains) was proved in
\cite{wlf62}. It is well known that Baer-Kaplansky Theorem cannot
be extended to torsion-free groups (of rank 1) since there are
infinitely many pairwise non-isomorphic torsion-free groups of
rank 1 whose endomorphism rings are isomorphic to $\Z$, \cite{Ar}.
However, similar results to Baer-Kaplansky Theorem hold for some
special classes of torsion-free groups, see \cite{BIS}. In the
setting of modules over complete valuation domains W. May proved a
theorem, \cite[Theorem 1]{May}, for reduced modules which are
neither torsion nor torsion-free and have a nice subgroup $B$ such
that $M/B$ is totally projective: if $M$ is such a module and $N$
is an \textsl{arbitrary} module such that they have isomorphic
endomorphism rings then $M\cong N$.

The main aim of this note is to prove a Baer-Kaplansky theorem for
arbitrarily modules over principal ideal domains (Theorem
\ref{canc-Z}): if $R$ is (commutative) principal ideal domain then
the correspondence (from the class of $R$-modules to the class of
rings) $$\Phi:G\mapsto \End_R(R\oplus G)$$ reflects isomorphisms
of endomorphism rings. Moreover, this property characterizes
principal ideal domains in the class of Dedekind domains: if $R$
is a Dedekind domain such that the correspondence $\Phi$ reflects
isomorphisms then $R$ is a PID. The restriction to Dedekind
domains is motivated by the fact that these domains have the
\textsl{cancellation property}, i.e. the endofunctor $R\oplus
-:\Modr R\to \Modr R$ on the category of all $R$-modules reflects
isomorphisms:

\begin{theorem} \cite[Proposition 3.6]{Lam}\label{canc-dedekind}
Let $R$ be a Dedekind domain. If $M$ and $N$ are two $R$-modules
such that $R\oplus M\cong R\oplus N$ then $M\cong N$.
\end{theorem}

We need this property in order to obtain that $\Phi$ reflects
isomorphisms (cf. Remark \ref{remark-cancellation}). However, in
order to obtain such a correspondence which reflects isomorphisms
the cancellation property is not enough, as it is proved in
Proposition \ref{prop-frtf} (in contrast with the similar problem
for
subgroup lattices, approached in 
\cite[Lemma 2]{Br-Ca}).

\section{A Baer-Kaplansky theorem}


The main result proved in this note is the following

\begin{theorem}\label{canc-Z}
Let $R$ be a Dedekind domain. The following are equivalent:
\begin{enumerate} \item The ring $R$ is a principal ideal domain;

\item If $G$ and $H$ are $R$-modules such that $G'=R\oplus G$ and
$H'=R\oplus H$ have isomorphic endomorphism rings then $G$ and $H$
are isomorphic.
\end{enumerate}
\end{theorem}

\begin{proof}
(1)$\Rightarrow$(2) Let $e$ and $f$ be the idempotents in
$\End_R(G')$ which are induced by the direct decomposition
$G'=R\oplus G$. Using the version for principal ideal domains of
\cite[Theorem 106.1]{Fu2}, we observe that there are isomorphisms
$$e\End_R(G')f\cong \Hom_R(G,R)$$ and
$$f\End_R(G')e\cong \Hom_R(R,G)\cong G.$$

If $\varphi:\End_R(G')\to \End_R(H')$ is an isomorphism then the
idempotents $\overline{e}=\varphi(e)$ and
$\overline{f}=\varphi(f)$ induce a direct decomposition
$H'=B\oplus K$, where $B=\overline{e}(H')$ and
$K=\overline{f}(H')$. By \cite[106(d)]{Fu2} there is an
isomorphism $\End_R(B)\cong R$. Moreover, as before, we have the
isomorphisms (of $R$-modules)
$$\Hom_R(K,B)\cong\overline{e}\End_R(H')\overline{f}\cong \Hom_R(G,R),$$
and $$\Hom_R(B,K)\cong\overline{f}\End_R(H')\overline{e}\cong
\Hom_R(R,G)\cong G.$$

\textsl{We claim that $B\cong R$.} Using this and Theorem
\ref{canc-dedekind} we obtain $H\cong K$, and we have $$H\cong
\Hom_R(R,K)\cong \Hom_R(B,K)\cong \Hom_R(R, G)\cong G.$$

In order to prove our claim, suppose that $B\ncong R$. Let
$\alpha:B\to R$ be an $R$-homomorphism. Since $R$ is a PID it
follows that $\Ima(\alpha)\cong R$, hence $\Ker(\alpha)\neq 0$.
Moreover, $\Ima(\alpha)$ is a projective module, hence we have a
direct decomposition $B\cong \Ker(\alpha)\oplus \Ima(\alpha)$. But
$\End(B)\cong R$ has no non-trivial idempotents, hence $B$ is
indecomposable. It follows that $\Ima(\alpha)=0$, hence
$\Hom_R(B,R)=0$.

If we consider the direct decomposition $H'=R\oplus H$ and the
canonical projection $\pi_R:H'\to R$, it follows that $B$ is
contained in $H$, the kernel of $\pi_R$. From $H'=B\oplus K$ we
obtain $H= (H\cap K)\oplus B$. Using the equalities $$K\oplus
B=R\oplus H=R\oplus (H\cap K)\oplus B$$ we deduce that $K\cong
R\oplus (H\cap K)$ (as complements for the direct summand $B$),
hence $K$ has a direct summand isomorphic to $R$. Therefore
$\Hom_R(G,R)\cong \Hom_R(K,B)$ has a direct summand isomorphic to
$B$. Since $R$ is commutative, $\Hom_R(G,R)$ is an $R$-module
which can be embedded as a submodule in the direct product $R^G$
of copies of $R$ (here we view $R^G$ as the set of all maps $G\to
R$, endowed with pointwise addition and scalar multiplication; see
\cite[Exercise 43.1]{Fu}). Therefore we can embed $B$ in $R^G$.
Since $B\neq 0$ it follows that we can find a projection
$\pi:R^G\to R$ such that $\pi(B)\neq 0$. This implies
$\Hom_R(B,R)\neq 0$, a contradiction, and it follows that $B\cong
R$.

\medskip

(2)$\Rightarrow$(1) Let $I$ be a non-zero ideal in $R$. Since $R$
is noetherian and integrally closed we can apply \cite[Theorem
I.3.7]{Fu-Sa} to conclude that $\End_R(I)\cong R$. Moreover, since
$I$ is invertible, we can use Steinitz isomorphism formula,
\cite[p.165]{Fu-Sa}. Therefore, for every positive integer $n$ we
have an isomorphism $(\oplus_{k=1}^{n-1} R)\oplus I^{n}\cong
\oplus_{k=1}^n I$, hence there are ring isomorphisms
$$\End_R((\oplus_{k=1}^{n-1} R)\oplus I^{n})\cong  \End_R(\oplus_{k=1}^n
I)\cong M_n(R)\cong \End_R(\oplus_{k=1}^n R).$$ If $n\geq 2$  we
obtain, from (2), that $(\oplus_{k=1}^{n-2} R)\oplus I^{n}\cong
\oplus_{k=1}^{n-1}R$. Using again the cancellation property of
$R$, Theorem \ref{canc-dedekind}, we conclude that $I^{n}$ is
principal for all $n\geq 2$. If $C(R)$ is the ideal class group
associated to $R$ and $[I]$ is the class of $I$ in this group, it
follows that $[I]^n=1$ for all $n\geq 2$, hence $[I]=1$. Then $I$
is principal and the proof is complete.
\end{proof}

\begin{remark}
From the above proof it follows that if $R$ is a principal ideal
domain then every ring isomorphism $\varphi:\End_R(R\oplus G)\to
\End_R(R\oplus H)$ induces a direct decomposition $R\oplus
H=B\oplus K$ with $B=\varphi(e)(R\oplus H)\cong R$ and
$(1-\varphi(e))(R\oplus H)=K\cong G$, where $e$ is the idempotent
such that $e(R\oplus G)=R$ and $(1-e)(R\oplus G)=G$. Since $B\cong
R$, it is not hard to see, using the same technique as in the
proof for the bounded case of \cite[Theorem 108.1]{Fu2}, that
$\varphi$ is induced by an isomorphism $\psi:R\oplus G\to R\oplus
H$. Therefore the above theorem can be viewed as an improvement of
\cite[Theorem 2.1]{wlf95} for the case of principal ideal domains.
\end{remark}

\begin{remark}
A class $\CC$ of modules is called \textsl{Baer-Kaplansly} if any
two of its modules are isomorphic whenever their endomorphism
rings are isomorphic as rings, \cite[p. 1489]{Iv-Va}. Therefore,
Theorem \ref{canc-Z} says that the class of modules over a
Dedekind domain $R$ which have a direct summand isomorphic to $R$
is a Baer-Kaplansky class if and only if $R$ is a principal ideal
domain. Similar results for other kind of rings were obtained in
\cite[Theorem 8]{Iv} for a similar class, respectively in
\cite[Theorem 4]{Iv-Va} for a particular class of modules over
FGC-rings.
\end{remark}

As a consequence of Theorem \ref{canc-Z} we obtain that locally
free modules over principal ideal domains are determined by their
endomorphism rings. This is also a consequence of \cite[Theorem
A]{wlf62}.

\begin{corollary}
If two locally free modules over a principal ideal domain have
isomorphic endomorphism rings then they are isomorphic.
\end{corollary}

\begin{remark}\label{remark-cancellation}
In the proof of Theorem \ref{canc-Z} we used the cancellation
property of the regular module $R$. If $R$ has not this property
(e.g. there are Dedekind-like domains without cancellation
property, \cite{Ha-Wi}) then there are two $R$-modules $G\ncong H$
such that $R\oplus G\cong R\oplus H$, hence $\End_R(R\oplus
G)\cong \End_R(R\oplus H)$. If we write these endomorphism rings
as matrix rings
$$\End_R(R\oplus G)=\left(\begin{array}{cc}
  \End_R(R) & \Hom_R(G,R) \\
  \Hom_R(R,G) & \End_R(G) \\
\end{array}\right)\cong \left(\begin{array}{cc}
  R & \Hom_R(G,R) \\
  G & \End_R(G) \\
\end{array}\right),$$ respectively $$\End_R(R\oplus H)\cong \left(\begin{array}{cc}
  R & \Hom_R(H,R) \\
  H & \End_R(H) \\
\end{array}\right),$$ we observe that the $(2,1)$-blocks in these
representations are isomorphic to $G$, respectively to $H$. These
two blocks are not isomorphic even the corresponding matrix rings
are isomorphic. It is obvious that in this case Theorem
\ref{canc-Z} is not valid.
\end{remark}

We will prove that we cannot replace in the implication
$(1)\Rightarrow (2)$ of Theorem \ref{canc-Z} the direct summand
$R$ by an arbitrary module wich have the cancellation property.
The following proposition shows that the property of the regular
module $R$ stated in Theorem \ref{canc-Z} is more stronger than
the usual
cancellation property (see \cite[Theorem B]{St}). 

\begin{proposition}\label{prop-frtf}
The following are equivalent for an indecomposable  torsion-free
abelian group $F\neq 0$ of finite rank: \begin{enumerate} \item If
$G$ and $H$ are abelian groups such that $\End(F\oplus G)\cong
\End(F\oplus H)$ then $G\cong H$;

\item $F\cong \Z$.
\end{enumerate}
\end{proposition}

\begin{proof}
(1)$\Rightarrow$(2) If $F$ is not isomorphic to $\Z$ then $F\cong
\Q$ or $F$ is a reduced abelian group which has no free direct
summands.

For the case $F\cong \Q$, we can choose $G$ and $H$ two
non-isomorphic subgroups of $\Q$ such that $\End(G)=\End(H)=\Z$.
It is not hard to see that both endomorphism rings $\End(F\oplus
G)$ and $\End(F\oplus H)$ are isomorphic to the matrix ring
$\left(\begin{array}{cc}
  \Q & 0 \\
  \Q & \Z \\
\end{array}\right)$, so $F$ does not verify the condition (1).

If $F$ is a reduced abelian group which has no free direct
summands then we can construct, using \cite[Theorem]{St}, two
(reduced) finite rank torsion-free groups $G$ and $H$ of the same
rank such that
$$\Hom(F,G)=\Hom(F,H)=\Hom(G,F)=\Hom(H,F)=0$$ and $$\End(G)=
\End(H)= \Z.$$ In this case both endomorphism rings $\End(F\oplus
G)$ and $\End(F\oplus H)$ are isomorphic to the ring
$\End(F)\times \Z$, so $F$ does not verify the condition (1).

\medskip

(2)$\Rightarrow$(1) This is a consequence of Theorem \ref{canc-Z}.
\end{proof}


\begin{remark}
There are also versions for the Baer-Kaplansky theorem proved for
automorphism groups, Jacobson radicals or for ring
anti-isomorphisms, \cite{Go}, \cite{HPS}, \cite{Le}, \cite{SSS}.
It would be nice to know if Theorem \ref{canc-Z} is still true if
we consider only automorphism groups or Jacobson radicals.
\end{remark}

\noindent{\textbf{Acknowldgements:}} I would like to thank
professor Luigi Salce for his help to prove the implication
(2)$\Rightarrow$(1) in Theorem \ref{canc-Z}.


\begin{thebibliography}{99}

\bibitem{Ar}  D. M. Arnold: Finite Rank Torsion Free Abelian Groups and
Rings, Lect. Notes in Math. Springer-Verlag, \textbf{931}, (1982).


\bibitem{BIS} E. Blagoveshchenskaya, G. Ivanov, P. Schultz: {\sl The Baer-
Kaplansky theorem for almost completely decomposable groups},
Contemporary Mathematics 273 (2001), 85--93.

\bibitem{Br-Ca} S. Breaz, G. C\u alug\u areanu:
{\sl Every Abelian group is determined by a subgroup lattice},
Stud. Sci. Math. Hung. 45 (2008), 135--137.

\bibitem{Go} A.L.S. Corner, B. Goldsmith, S.L. Wallutis: {\sl Anti-isomorphisms
and the failure of duality}, in ``Models, modules and abelian
groups. In memory of A. L. S. Corner'', Walter de Gruyter, 2008,.
315--323.

\bibitem{Fu} L. Fuchs: Infinite Abelian Groups I, Academic Press, 1970.

\bibitem{Fu2}
L. Fuchs: {Infinite Abelian Groups II}, Academic Press, 1973.

\bibitem{Fu-Sa} L. Fuchs, L. Salce: Modules over non-Noetherian
domains, Mathematical Surveys and Monographs 84, American
Mathematical Society, 2001.

\bibitem{Iv} G. Ivanov: {\sl
Generalizing the Baer-Kaplansky theorem}, J. Pure Appl. Algebra
133 (1998), 107--115.

\bibitem{Iv-Va} G. Ivanov, P. V\'amos: {\sl
A characterization of FGC rings}, Rocky Mt. J. Math. 32 (2002),
1485--1492.

\bibitem{HPS} J. Hausen, J., C. Praeger, P. Schultz: {\sl Most abelian p-groups
are determined by the Jacobson radical of their endomorphism
rings}, Math Z. 216 (1994), 431--436.

\bibitem{Ha-Wi} W. Hassler, R. Wiegand: {\sl Direct sum cancellation for
modules over one-dimensional rings}, J. Algebra 283 (2005),
93--124.


\bibitem{Lam} T.Y. Lam:
{\sl A crash course on stable range, cancellation, substitution
and exchange}, J. Algebra Appl. 3 (2004), 301--343.

\bibitem{Le}
H. Leptin: {\sl Abelsche p-Gruppen und ihre
Automorphismengruppen}, Math. Z. 73 (1960), 235–-253.

\bibitem{Li}
W. Liebert: {\sl Endomorphism rings of free modules over principal
ideal domains}, Duke Math. J. 41 (1974), 323-328.

\bibitem{May} W. May:
{\sl Isomorphism of endomorphism algebras over complete discrete
valuation rings}, Math. Z. 204 (1990), 485--499.

\bibitem{RW}
F. Richman, E.A. Walker: {\sl Primary abelian groups as modules
over their endomorphism rings}, Math. Z. 89 (1965), 77--81.

\bibitem{SSS} P. Schultz, A. Sebeldin, A.L. Sylla:
{\sl Determination of torsion Abelian groups by their automorphism
groups}, Bull. Aust. Math. Soc. 67 (2003), 511--519.

\bibitem{St} J. Stelzer: {\sl A cancellation criterion for finite-rank
torsion-free Abelian groups}, Proc. Am. Math. Soc. 94 (1985),
363--368.

\bibitem{wlf95}
K. Wolfson: {\sl Isomorphisms Between Endomorphism Rings of
Modules}, Proc. Am. Math. Soc. 123 (1995), 1971--1973.

\bibitem{wlf89}
K. Wolfson: {\sl Anti-Isomorphisms of Endomorphism Rings of
Locally Free Modules}, Math.Z. 202 (1989), 151-–159.

\bibitem{wlf62}
K.G. Wolfson: {\sl Isomorphisms of the endomorphism rings of
torsion-free modules}, Proc. Amer. Math. Soc. 14 (1963), 589--594.









\end{thebibliography}
 \end{document}